\documentclass{amsart}
\newtheorem{prop}{Proposition}[section]
\newtheorem{lema}[prop]{Lemma}
\newtheorem{teo}[prop]{Theorem}

\newtheorem{question}[prop]{\sc Question}

\usepackage{todonotes}
\newcommand{\pe}{\hat{\otimes}_{\epsilon}}
\newcommand{\ppi}{\hat{\otimes}_{\pi}}
\title{Characterization of dual mixed volumes via polymeasures}

\author{Carlos H. Jim\'enez}
\address{Departamento de An\'alisis Matem\'atico \\
Facultad de Matem\'aticas \\ Universidad Complutense de Madrid \\
Madrid 28040}

\email{carloshugo@us.es}

\author{Ignacio Villanueva}

\email{ignaciov@mat.ucm.es}

\thanks{C. H. Jim\'enez is supported by Spanish Ministry of Economy and Competitiveness (MINECO) project MTM2012-30748 and also by Mexico's National Council for Science and Technology (Conacyt) postdoctoral grant 180486. I. Villanueva is partially supported by MINECO (grant MTM2011-26912)}

\begin{document}

\begin{abstract}
We prove a characterization of the dual mixed volume in terms of functional properties of the polynomial associated to it. To do this, we use tools from the theory of multilinear operators on spaces of continuos functions. Along the way we reprove, with these same techniques, a recently found characterization of the dual mixed volume.

\end{abstract}



\maketitle

\section{Introduction}
The Brunn-Minkowski theory is one of the cornerstones of modern Convex Geometry. This theory can trace back its origin to Minkowski's efforts to study the relation between two basic mathematical concepts on the set of Convex Bodies: Minkowski (or vector) addition and volume. Among the main objects of study on this theory we find the mixed volumes (see next section for definitions) which allow us to study some other very important concepts such as volume, mean width and surface area measure all within a consolidated framework.  For a comprehensive exposition on this we refer to \cite{Schb,Gardb}. Several concepts from the Brunn-Minkowski theory have been successfully extended in many ways. These extensions can reach other areas of mathematics, like those providing similar versions for log-concave functions or those of purely geometric nature. From these extensions, in this note we mainly deal with the dual mixed volume. Dual mixed volumes lie at the very core of what is now called the dual Brunn-Minkowski theory. With the development of this theory  came along several new tools that have been successfully applied in many different areas such as Integral Geometry, Geometric tomography and Local theory of Banach spaces among others. It should be mention as well, that it played a key role in the solution of Busemann-Petty problem \cite{Gard2,GarKolSchl,Lut_mv,Zha}.

Recent years have witnessed an increasing interest in the characterization of several  important parameters and operations on convex (or star) sets. The purpose is to obtain characterizations based on different fundamental properties of the parameters. Some of these properties are invariance (or covariance) with respect to basic transformations, symmetry, positivity or additivity. 

In the papers \cite{DuGaPe, MiSc}, the authors study different characterizations of the mixed volume and the dual mixed volume by their respective functional properties. 

In this note, we continue this study applying known results from the theory of multilinear operators on spaces of continuous functions.  In our main result, Theorem \ref{main}, we  provide a set or equivalent conditions that characterize the mixed volume. One of these conditions is stated in terms of the polynomial associated to the dual mixed volume, which turns out to be the $n$-dimensional volume.
As a byproduct we also reprove the functional characterization of the dual mixed volume that appeared in \cite{DuGaPe}.   

The basic idea of the connection is the following: it seems natural to try to characterize the dual mixed volume in terms of its separate additivity with respect to the radial sum. An additive function on star bodies induces an additive function on their radial functions. This, in turn, induces an additive function on $C(S^{n-1})^+$, the positive cone of the space of continuous functions on the unit sphere in $\mathbb R^n$. With a little extra help additive functions become linear. Therefore, we run into the study of multilinear forms on $C(S^{n-1})$, the space of continuous functions on $S^{n-1}$.  The situation is analogous for the case of star sets. In that case one runs into the multilinear forms on $B(\Sigma_n)$, the space of functions on $S^{n-1}$ which are the uniform limit of simple Borel functions.  

The dual spaces $C(S^{n-1})^*$ and $B(\Sigma_n)^*$ are well described by regular measures and additive measures, respectively, on $\Sigma_n$, the Borel $\sigma$-algebra of $S^{n-1}$. Similarly, multilinear forms on $C(S^{n-1})$ or $B(\Sigma_n)$ are well described by {\em polymeasures},  functions $\gamma:\Sigma_n\times \cdots \times \Sigma_n:\longrightarrow \mathbb R$ which are separately measures. 

The theory of polymeasures, or multilinear operators on $C(K)$ and $B(\Sigma)$ spaces, has been studied for some time now \cite{Do-VIII,BoVi,BoPeVi}. In particular, there are two results already in the literature which will be needed. 

First of all, it will be important to  understand  when can a polymeasure be extended to a measure in $\Sigma\otimes \cdots \otimes \Sigma$, the product $\sigma$-algebra. The answer to this for the case of Radon polymeasures appeared for the first time in \cite{BoVi}, although the particular case of bimeasures, or bilinear operators, had already been studied in \cite{KM}. The answer for bounded additive polymeasures is simpler and follows along the same lines. 

Also, it is important for our purposes to understand when  the action of a polymeasure $\gamma:\Sigma_n\times \cdots \times \Sigma_n:\longrightarrow \mathbb R$, given by $\int (f_1(t_1), \ldots, f_m(t_n)) d\gamma(t_1, \ldots, t_m)$,  can be described by means of a measure $\mu:\Sigma_n\longrightarrow \mathbb R$ via the integral $\int f_1(t)\cdots f_m(t) d\mu(t)$. The answer to this appeared for the first time in \cite{OA} and \cite{BeLaLl} for the case of Radon polymeasures. Again, the case of bounded additive measures is simpler and is essentially contained in the previous one. 

For completeness, we state these results in the next section.

\section{Notation and previous results}

We will need several definitions and results from the theory of polymeasures and multilinear operators in spaces of continuous functions as well as from Convex Geometry.

As usual, $S^{n-1}$ stands for the Euclidean unit sphere. We denote by $\Sigma_n$ its Borel $\sigma$-algebra. For $K\subset\mathbb R^n$ compact, its support function $h_K$ is defined as $$h_K(x)=\max\{x\cdot y\ :\ y\in K\}.$$ The support function uniquely characterizes the set $K$. The Minkowski addition of sets $K$ and $L$ is defined as the set $K+L$ satisfying $h_{K+L}(x)=h_K(x)+h_L(x)$. We denote the space of convex bodies (compact, convex set with non-empty interior) by $\mathcal{ K}^n$. For $K_1,...,K_m\in\mathcal{ K}^n$ and $\lambda_1,...,\lambda_m\geq 0$, the volume $$|\lambda_1K_1+...+\lambda_mK_m|=\sum_{1\leq i_1,...,i_n\leq m}\lambda_{i_1}...\lambda_{i_n}V(K_{i_1},...,K_{i_n})$$ is a polynomial on the variables $\lambda_1,...,\lambda_m$ where the coefficients $V(K_{i_1},...,K_{i_n})$ are called mixed volumes. The previous non-trivial fact is a classical theorem in the field due to Minkowski.  
 
A set $L$ is star shaped at $0$ if every line through $0$ that meets $L$ does so in a (possibly degenerate) line segment. We denote by $\mathcal S^n$ the set of the star sets. 

Given a star set $L$, we define its {\em radial function} $\rho_L$ by

\begin{eqnarray*}
\rho_L(x)= \left \{ \begin{array}{cl}
      \displaystyle{ \max\{c\, : \, cx\in L\}}
        & \displaystyle{\quad \mbox{if } L\cap l_x \not = \emptyset }\\
        \noalign{\bigskip}
 0 \quad  & \quad \mbox{otherwise} 
 \end{array} \right.
\end{eqnarray*}

Clearly, radial functions are totally characterized by their restriction to $S^{n-1}$, so from now on we consider them defined on $S^{n-1}$.

A star set $L$ is called a star body if and only if $\rho_L$ is continuous. We denote by $\mathcal S_0^n$ the set of star bodies. Conversely, given a positive and continuous function $f: S^{n-1}\longrightarrow [0,\infty)$, 
it can be considered as the radial function of a star body $L_f$. 

Given two sets $L,M\in \mathcal S^n$, we define their {\em radial sum} as the star set $L\tilde{+}M$ whose radial function is $\rho_L+\rho_M$.  

For $L_1,...,L_n\in \mathcal S_0^n$ their dual mixed volume is defined as

$$\tilde{V}(L_{1},...,L_{n})=\frac{1}{n}\int_{S^{n-1}}\rho_{L_{1}}(u)...\rho_{L_{n}}(u)du.$$

In \cite{Lut_mv1} Lutwak came up with the following analogous result to Minkowski's theorem for mixed volumes mentioned earlier: given  for sets $L_1,...,L_m\in \mathcal S_0^n$ and $\lambda_1,...,\lambda_m>0$ the volume of the radial sum verifies $$|L_1\tilde{+}...\tilde{+}L_m|=\sum_{1\leq i_1,...,i_n\leq m}\lambda_{i_1}...\lambda_{i_n}\tilde{V}(L_{i_1},...,L_{i_n}).$$

That is, $|L_1\tilde{+}...\tilde{+}L_m|$ 
is an homogeneous polynomial of degree $n$  in the variables $\lambda_{1},...,\lambda_{m}$ where each coefficient $\tilde{V}(L_{i_1},...,L_{i_n})$ depends only on the bodies $L_{i_1},...,L_{i_n}$.

We state now the definitions and results from the theory of polymeasures that we will need for our results. 

Let $\Sigma _i\  (1\le i\le m)$ be $\sigma$-algebras
(or simply algebras) of subsets on some
non void sets $S_i$.
A function $\gamma:{\Sigma }_{1} \times \cdots \times {\Sigma }_{m} \longrightarrow \mathbb R$
is a (countably additive) $m$-{\em polymeasure} if it is separately
(countably) additive. 

Same as in the case $m=1$, we can define the \begin{em}
variation \end{em} of a polymeasure
as the set function $$v(\gamma )
:{\Sigma }_{1} \times \cdots \times {\Sigma }_{m} \longrightarrow
[0,+\infty]$$ given by

$$v(\gamma )(A_{1}, \ldots ,A_{m}) =\sup
\left\{\sum_{j_{1}=1}^{n_{1}} \cdots \sum_{j_{m}=1}^{n_{m}}
\left| \gamma (A_1^{j_1}, \ldots A_m^{j_m} )\right| \right\},$$
where the supremum is taken over all the finite $\Sigma
_i$-partitions $(A_i^{j_i})_{j_{i}=1}^{n_{i}}$ of $A_{i}$ ($1\leq
i \leq m$). 

We can define also its
\begin{em} semivariation \end{em}
$$\|\gamma \|
:{\Sigma }_{1} \times \cdots \times {\Sigma }_{m} \longrightarrow
[0,+\infty]$$ by

    $$ \|\gamma \| (A_{1}, \ldots , A_{m}) = \sup \left \{ \left |
    \sum_{j_{1}=1}^{n_{1}} \cdots \sum_{j_{m}=1}^{n_{m}} a_1^{j_1}
 \ldots a_m^{j_m} \gamma (A_1^{j_1}, \ldots, A_m^{j_m} )\right
    | \right\}$$
where the supremun is taken over all the finite $\Sigma
_i$-partitions $(A_i^{j_i})_{j_{i}=1}^{n_{i}}$ of $A_{i}$ ($1\leq
i \leq m$), and all the collections $(a_i^{j_i})_{j_i=1}^{n_i}$
contained in the unit ball of the scalar field. 


If $\gamma $ has finite semivariation, an elementary integral
$\int (f_1,f_2,\ldots f_m)\>d\gamma $ can be defined, where $f_i$
are bounded, $\Sigma_i$-measurable scalar functions, just taking
the limit of the integrals of $m$-uples of simple functions (with
the obvious definition) uniformly converging to the $f_i$'s (see
\cite{Do-VIII}).

Polymeasures can be used to represent continuous multilinear forms in spaces of continuous functions. We need some notation. 
Given a Hausdorff compact set $S$, $\Sigma_S$ is the $\sigma$-algebra of its Borel sets. 
$S(S)$ is the set of Borel simple functions with support in $\Sigma_S$. $B(S)$ is the completion of $S(S)$ with respect to the supremum norm.  $C(S)$ is the space of the continuous functions defined on the compact set $S$, endowed with the supremum norm. 

$C(S)$ is naturally contained in $B(S)$, and, in turn $B(S)$ is naturally contained in the bidual space $C(S)^{**}$, the inclusion being given by $$\langle g, \mu\rangle=\int g d\mu$$ for every $g\in B(S)$, $\mu\in C(S)^*$. In particular, this implies that every continuous linear form $T:C(S)\longrightarrow \mathbb R$ can be extended by weak$^*$ continuity over $C(S)^{**}$ to a continuous linear form $\overline{T}:B(S)\longrightarrow \mathbb R$. The same fact remains true for multilinear forms, where measures are replaced by polymeasures. 

\begin{teo}[\cite{BoVi}]\label{representacion}
Let  $S$ be a compact Hausdorff
space and $\Sigma$ its Borel $\sigma$-algebra. Every $m$-linear continuous form $T:
C(S)\times \cdots \times C(S)\longrightarrow \mathbb R$ has a unique
\begin{em}
representing polymeasure
\end{em} $\gamma : \Sigma\times \cdots \Sigma
\rightarrow \mathbb R$ with finite semivariation. $T$ and $\gamma$ are related by the formula 
    $$T(f_1,\ldots ,f_m) = \int (f_1,\ldots ,f_m)\> d\gamma \ \hbox{
 for every } f_1, \ldots, f_m \in C(S).$$
$\gamma$ is separately countably additive and regular .

Conversely, if $\gamma : \Sigma\times \cdots \Sigma
\rightarrow \mathbb R$ is a regular countably additive polymeasure then it has finite semivariation and $$T(f_1,\ldots ,f_m) = \int (f_1,\ldots ,f_m)\> d\gamma$$ defines a continuous multilinear form $T:C(S)\times \cdots \times C(S)\longrightarrow \mathbb R$ that satisfies $\|T\|=\|\gamma\|$.

Moreover, there exists a unique extension $T^{**}:C(S)^{**}\times \cdots \times C(S)^{**}\longrightarrow \mathbb R$ of $T$ that is separately weak$^*$ continuous. 

This extension can be defined the following way: For $(z_1, \ldots, z_m)\in (C(S)^{**}\times \cdots \times C(S)^{**})$ we can choose, for every $1\leq i \leq m$, a net  $(f_{\alpha_i})_{\alpha_i}\subset C(S)$ such that $\lim_{\alpha_i} f_{\alpha_i}=z_i$ in the weak$^*$ topology. Then $$T^{**}(z_1,\ldots, z_m)=\lim_{\alpha_1}\cdots \lim_{\alpha_m} T(f_{\alpha_1}, \ldots, f_{\alpha_m}).$$

In particular, the restriction to the product of the $B(S)$'s defines a continuous multilinear form $$\overline{T}:B(S)\times \cdots \times B(S)\longrightarrow \mathbb R$$ such that, for all Borel sets $(A_1,\ldots, A_m)\in \Sigma_S\times \cdots \times \Sigma_S$, $$\overline{T}(\chi_{A_1}, \ldots, \chi_{A_m})=\gamma(A_1, \ldots, A_m).$$
\end{teo}

Given a polymeasure $\gamma$ we can consider the set function
$\gamma _m$ defined on the semi-ring of all measurable rectangles
$A_1\times\cdots \times A_m\ (A_i\in \Sigma)$ by
    $$
\gamma _m(A_1\times\cdots \times A_m) := \gamma (A_1,\ldots ,A_m)
    $$
It follows, for instance, from \cite[Prop. 1.2]{DiMu}, that $\gamma _m$ is
finitely additive and then it can be uniquely extended to a
finitely additive measure on the algebra $a(\Sigma\times \cdots
\times \Sigma)$ generated by the measurable rectangles. In
general, this finitely additive measure is not bounded and therefore it can not be extended to the  $\sigma$-algebra $\Sigma\otimes \cdots \otimes
\Sigma$ generated by $\Sigma\times \cdots \times \Sigma$.



\smallskip

The extension of polymeasures to measures defined on the product $\sigma$-algebra is related to the extension of multilinear operators from the projective tensor product of Banach spaces to the injective tensor product. 

We refer to \cite{DeFl} for the definition of projective and injective tensor product of Banach spaces. Following the usual notation, we refer to them as $X\ppi X$ and $X\pe X$ respectively. 
We recall that $C(S\times S)$ is isometric to the injective tensor product  $C(S)\pe C(S)$.

A continuous multilinear form $T:X\times \cdots \times X\longrightarrow \mathbb R$ on the product of Banach spaces induces a continuous form $T:X\ppi \cdots \ppi X\longrightarrow \mathbb R$. 

In general, this form will not be continuous when considered defined on the injective tensor product. 

The continuity of a multilinear form $T:C(S)\times \cdots \times C(S)\longrightarrow \mathbb R$ when considered defined on the injective tensor product $C(S)\pe \cdots \pe C(S)=C(S\times \cdots \times S)$ is equivalent to the extendability of the associated polymeasure $\gamma$ to a measure on the Borel sets of $S\times \cdots \times S$. This is the content of the next result, proved in \cite{BoViC(K)}. As mentioned there, the result was proven for the special case of bilinear forms in \cite{KM}, but it does not seem possible to extend the  proof techniques of \cite{KM} to $m>2$.

\begin{teo}\label{teo1}
Let $T:C(S)\times\cdots \times C(S)\longrightarrow \mathbb R$ be an $m$-linear form with
representing polymeasure $\gamma$.

Then the following are equivalent:

a) $v( \gamma)<\infty$.

b) $T:C(S)\pe \cdots\pe C(S)\longrightarrow \mathbb R$ is continuous. 

c) $\gamma$ can be extended to a regular measure $\mu$ on the $\sigma$-algebra of the Borel sets of $S\times \cdots \times S$. 

d) $\gamma$ can be decomposed as the sum of a positive and a negative polymeasure.
\end{teo}

Let us also recall that, if $S$ is compact metrizable, then
$\Sigma_S\otimes \cdots \otimes \Sigma_S=\Sigma_{S\times \cdots \times S}$, where $\Sigma_L$ denotes the Borel $\sigma$-algebra of a set $L$ and 
$\Sigma_S\otimes \cdots \otimes \Sigma_S$ denotes the product
$\sigma$-algebra, i.e., the smallest $\sigma$-algebra that
contains the rectangles $(A_1\times \cdots \times A_m)\subset
S^m$, with $A_i\in \Sigma_S$ ($1\leq i \leq m$) (see \cite[7.6.2
and 7.1.12]{Co}).

 It follows from the previous paragraph that we can identify $\Sigma_n\otimes \cdots \otimes \Sigma_n$ with $\Sigma_{(S^{n-1})^m}$, the $\sigma$-algebra of the Borel sets of $S^{n-1}\times \cdots \times S^{n-1}$. 

Before we state our last result from the theory of polymeasures, we need one more definition. 

Given a Banach lattice $X$ and a Banach space $Y$ (in our case $Y$ will be the scalar field), a function $\varphi:X\longrightarrow  Y$ is called {\it orthogonally
additive} if, for every $f,g\in X$ with disjoint support,
$\varphi(f+g)=\varphi(f)+\varphi(g)$. Orthogonally additive
functions and their representations have been studied by several
authors since the sixties (\cite{DreOr3},
\cite{FrKa1}, \cite{FrKa2}). 

The following result was proven in \cite{OA} and, independently, in \cite{BeLaLl}.

\begin{teo}\label{TeoOA}
Let $P:C(S)\longrightarrow \mathbb R$ be an orthogonally additive $n$-
homogeneous polynomial with associated multilinear form $T$.
Then, there exists a continuous linear form $\varphi\in C(S)^*$,
with associated measure $\nu:\Sigma\longrightarrow \mathbb R$, such that
$\|\varphi\|=\|T\|$ and such that, for every $f\in C(S)$,
    $$P(f)=\varphi(f^n)=\int_S f^n d\nu.$$
\end{teo}

\section{Characterizing the dual mixed volume}
In this section we state and prove our characterizations of the dual mixed volume. Following \cite{MiSc} and \cite{DuGaPe}, our approach is to characterize the dual mixed volume based on its functional properties. 

When working with radial sets, the natural sum is the radial sum $\tilde{+}$ defined above. Therefore, whenever we say in this note that a function defined on $\mathcal S_0^n$ is additive we will mean additive with respect to the radial sum.

One of the properties we will use to characterize the dual mixed volume is its separate additivity. Let us consider a cone $V$ and an additive function $f:V\longrightarrow \mathbb R$. Then, for every $L\in V$ the function $f_L:\mathbb R^+\longrightarrow \mathbb R$ defined by $f_L(\lambda)=f(\lambda L)$ is additive and can be extended trivially to an additive function $\tilde{f}_L:\mathbb R\longrightarrow \mathbb R$ by $\tilde{f}_L(-\lambda)=-f_L(\lambda)$ 

Real additive functions are either homogeneous or quite pathological. For instance, given an additive function $f:\mathbb R\longrightarrow \mathbb R$ which is not homogeneous, its graph is dense in $\mathbb R^2$. So, an additive function $f:\mathbb R\longrightarrow \mathbb R$, or $f:\mathbb R^+\longrightarrow \mathbb R$ becomes homogeneous if it takes values in $[0,\infty)$, or it is continuous in one point, monotonic on any interval, or bounded on any interval. 

For this reason, we will require our functions to be additive and positive homogeneous, rather than just additive. We feel this improves the clarity of the reasonings.

\smallskip

Thus, the starting point for our characterization will be a  separately additive and positively homogeneous $F:\mathcal S_0^m \longrightarrow \mathbb R$. 

One such $F$ induces a separately additive and positively homogeneous application  $$T_F:C(S^{n-1})^+\times \cdots \times C(S^{n-1})^+\longrightarrow \mathbb R$$ defined by 
$$T_F(f_1, \ldots, f_m)=F(L_{f_1}, \cdots, L_{f_m}),$$
and $T_F$ can be extended to a separately linear application $T_F:C(S^{n-1})\times \cdots \times C(S^{n-1})\longrightarrow \mathbb R$ in the natural way: for $(f_1, \ldots, f_m)\in C(S^{n-1})^m$, consider the decomposition $f_i=f_i^+-f_i^-$, with $f_i^+, f_i^-\in C(S^{n-1})$ and define $$T_F(f_1,\ldots, f_m)=T_F(f_1^+-f_1^-,\ldots, f_m^+-f_m^-)=$$  $$=T_F(f_1^+,f_2^+,\ldots, f_m^+)-T_F(f_1^+, f_2^+,\ldots, f_m^-)+\cdots + (-1)^m T_F(f_1^-, f_2^-,\ldots, f_m^-).$$

In general, $T_F$ is not separately bounded (equivalently, continuous). To see an example of this, consider $m=1$ and $T_F:C(S^{n-1}) \longrightarrow \mathbb R$ any linear not continuous application.

But even if $T_F$ is separately bounded, it could well be that it can not be extended to a linear bounded application $T:C(S^{n-1}\times \cdots \times S^{n-1})\longrightarrow \mathbb R$. As mentioned in the introduction, the reason for this is that $T_F$ need not be a priori continuous for the injective topology in $C(S^{n-1})\pe \cdots \pe C(S^{n-1})$. 

Clearly, we are not interested in unbounded functions on the product of the unit spheres as candidates for representing the dual mixed volume. Therefore, we first characterize those separately linear functions which can be extended to bounded linear applications on the product of the unit spheres.  First we give a condition on $F$ that characterizes the fact that $T_F$ is separately continuous, equivalently continous as a multilinear form.

\begin{prop}\label{primera}
Let $F:\mathcal S_0^m \longrightarrow \mathbb R$ be separately additive and positively homogeneous. Let us also suppose that $F$ is separately bounded, in the following sense: 

For every choice of $(L_2, \ldots, L_m)\in K_0^{m-1}$ there exists a constant $C$ such that $|F(L, L_2, \ldots, L_m)|\leq C$ for every $L$ totally contained in the unit ball of $\mathbb R^n$, and the same condition holds when  the role of the first variable is played by any other variable. 

Then, there exists a separately regular polymeasure $\gamma:\Sigma_n\times \cdots \times \Sigma_n \longrightarrow \mathbb R$ such that, for every $(L_1,\ldots, L_m)\in \mathcal{S}_0^m$,  $$F(L_1,\ldots, L_m)=\int (\rho_{L_1}(t_1), \ldots, \rho_{L_m}(t_m))d\gamma(t_1,\ldots, t_m). $$

Conversely, given a separately regular polymeasure $\gamma:\Sigma_n\times \cdots \times \Sigma_n \longrightarrow \mathbb R$ it induces a separately additive, positively homogeneous and bounded application $F:\mathcal S_0^m \longrightarrow \mathbb R$ by  

$$F(L_1,\ldots, L_m)=\int (\rho_{L_1}(t_1), \ldots, \rho_{L_m}(t_m))d\gamma(t_1,\ldots, t_m),$$ for every $(L_1,\ldots, L_m)\in \mathcal S_0^m$,
\end{prop}

\begin{proof}
Let us consider the function $T_F$ defined above. The separate boundedness condition implies that for every choice of positive functions $( f_2, \ldots, f_m)\in C(S^{n-1})^+\times \cdots \times C(S^{n-1})^+$, there exist a constant $C$ such that the induced mapping $T_F(\cdot, f_2, \ldots, f_m):C(S^{n-1})\longrightarrow \mathbb R$ verifies  $$|T(f_1, f_2, \ldots, f_m)|\leq C$$ for every $f_1\in C(S^{n-1})^+$ with $\|f_1\|\leq 1$. 

Given $f\in C(S^{n-1})$ with $\|f\|\leq 1$, $f$ can be decomposed as $f=f^+-f^-$, where $f^+, f^-\in C(S^{n-1})^+$ and $\|f^+\|\leq 1$, $\|f^-\|\leq 1$.

Therefore, $$|T(f_1, f_2, \ldots, f_m)|\leq 2C$$ for every $f_1\in C(S^{n-1})$ with $\|f_1\|\leq 1$. 

Now, given $(f_2, \ldots, f_m)\in C(S^{n-1})^{m-1}$, we can decompose each of them as the difference of two positive functions, and standard reasonings show that $T_F(\cdot, f_2, \ldots, f_m)$ is continuous. Multilinear separately continuous functions are jointly continuous, therefore  $T_F:C(S^{n-1})\times \cdots \times C(S^{n-1})\longrightarrow \mathbb R$ is continuous. 

Now, we just have to apply Theorem \ref{representacion} to obtain $\gamma$. 

The converse statement follows immediately from  Theorem \ref{representacion}.

\end{proof}

As mentioned before, in general, $\gamma$ defined as above can not be extended to a regular measure $\mu: \Sigma_n\otimes \cdots \otimes \Sigma_n\longrightarrow \mathbb R$. Therefore, as a consequence of this and Proposition \ref{primera}, not every separable additive, homogeneous and bounded application $F:S_0^m\longrightarrow \mathbb R$ can be represented by a measure on $\Sigma_n\otimes \cdots \otimes \Sigma_n$.

In the next proposition we give a condition which characterices the extendability of $\gamma$, and, hence, the representability of $F$ by a measure.  Note that we do not need to assume a priori that $F$ is separately bounded. 

\begin{prop}\label{segunda}
Let $F:\mathcal S_0^m \longrightarrow \mathbb R$ be separately additive and positively homogeneous. Then the following are equivalent:

\begin{enumerate}
\item There exist two separately additive positively homogeneous functions $F^+:\mathcal S_0^m \longrightarrow [0,\infty)$, $F^-:\mathcal S_0^m \longrightarrow [0, \infty)$ such that $F=F^+-F^-$. 

\item There exists a radon measure on $\nu: \Sigma_{(S^{n-1})^m}\longrightarrow \mathbb R$ such that $$F(L_1,\ldots, L_m)=\int \rho_{L_1}\otimes \cdots \otimes \rho_{L_m} (t_1,\ldots, t_m)d\nu(t_1,\ldots, t_m). $$

\end{enumerate}
\end{prop}

\begin{proof}
Let us suppose (1). We use $T_{F^*}$ for any of $T_{F^+}$ or $T_{F^-}$. 

Let us first see that $T_{F^*}$ is separately monotone in the following sense: If $f\geq g\geq 0$, then for every $f_2, \ldots, f_m\in C(S^{n-1})^+$ $$T_{F^*}(f, f_2, \ldots, f_m)\geq T_{F^*}(g, f_2, \ldots, f_m)$$
This follows from the separate additivity of $F^*$ by $$T_{F^*}(f, f_2, \ldots, f_m)= T_{F^*}(g, f_2, \ldots, f_m)+ T_{F^*}(f-g, f_2, \ldots, f_m)\geq T_{F^*}(g, f_2, \ldots, f_m).$$

By standard reasonings, it is easy to prove now that $T_{F^*}$ is monotone in the following sense: If $f_i\geq g_i\geq 0$ ($1\leq i \leq m$), then  $$T_{F^*}(f_1,  \ldots, f_m)\geq T_{F^*}(g_1,  \ldots, g_m).$$

To see that $T_{F^*}$ is continuous, we just need to prove that it is bounded on norm one functions. Decomposing functions into positive and negative part, if suffices to prove that it is bounded for positive norm one functions. If this was not the case, there would exist a sequence $(f_1^i, \ldots, f_m^i)_{i\in \mathbb N}\subset C(S^{n-1})^+\times \cdots \times C(S^{n-1})^+$, with $\|f_j^i\|\leq 1 $ for every $i,j$ such that, for every $i\in \mathbb N$,  
$$T_{F^*}(f_1^i, \ldots, f_m^i)\geq i.$$

For $1\leq j \leq m$  we define $\varphi_j=\sup_i f_j^i$. Then $\varphi_j\in C(S^{n-1})^+$, $\|\varphi_j\|\leq 1$, and we have 
$$F^*(L_{\varphi_1},  \ldots, L_{\varphi_m})\geq i$$ for every $i\in \mathbb N$, a contradiction with the fact that $F^*$ is bounded. 

Therefore, according to Theorem \ref{representacion}, there exists a separately regular polymeasure $\gamma_{F^*}:\Sigma(S^{n-1})\times \cdots \times \Sigma(S^{n-1})\longrightarrow \mathbb R$ such that 

$$T_{F^*}(f_1, \ldots, f_k)=\int (f_1(t_1), \ldots, f_m(t_m) )d\gamma_{F^*}(t_1, \ldots, t_m).$$

Let us see that $\gamma_{F^*}$ is positive: Pick $(A_1,\ldots, A_m)\in  \Sigma_n\times \cdots \times \Sigma_n$ and  $\epsilon>0$. Applying \cite[Lemma 2.5]{OA} we know of the existence of compact sets $K_i\subset A_i$ ($1\leq i \leq m$) such that $$|\gamma_{F^*}(A_1,\ldots, A_m)-\gamma_{F^*}(K_1, \ldots, K_m)|<\epsilon.$$ 

Applying again the same lemma, we obtain the existence of open sets $G_i\supset K_i$  ($1\leq i \leq m$) such that

$$|\gamma_{F^*}(G_1,\ldots, G_m)-\gamma_{F^*}(K_1, \ldots, K_m)|<\epsilon.$$

We can now use Urysohn's Lemma to find, for every $1\leq i \leq m$, a function $f_i \in C(S^{n-1})^+$ with $supp f_i\subset G_i$ and $f_i(t)=1$ for every $t\in K_i$. 

We now have $$\gamma_{F^*}(A_1,\ldots, A_m)>\gamma_{F^*}(K_1, \ldots, K_m)-\epsilon >$$ $$>\gamma_{F^*}(G_1,\ldots, G_m)-2\epsilon\geq \int (f_1(t_1), \ldots, f_m(t_m) d\gamma_{F^*}(t_1, \ldots, t_m) -2\epsilon \geq -2\epsilon.$$

This proves that $\gamma_{F^*}$ is positive. Therefore, $\gamma=\gamma_{F^+}-\gamma_{F^-}$ is the sum of a positive and a negative polymeasure.  Now, Theorem \ref{teo1} proves that $T_F$ can be represented by a regular measure $\nu$ defined on $\Sigma_{(S^{n-1})^m}$, and  (2) follows. 

\smallskip

Conversely, suppose that $F$ can be represented by a measure $\nu$ as in (2). It is very easy to check that in that case $\nu$ also represents $T_F$. The decomposition $\nu=\nu^+-\nu^-$  (see \cite{Co}) induces now the decomposition $F=F^+-F^-$ with the required properties. 
\end{proof}

Before we can state our main theorem, we isolate a technicality of the proof for clarity in the presentation. It is a direct translation of \cite[Lemma 2.6]{OA}.

\begin{lema}\label{continuas}
Let  $T_F:C(S^{n-1})\times \cdots\times C(S^{n-1})\longrightarrow \mathbb R$ be a continuous multilinear form and let $\gamma_F:\Sigma^n\times \cdots \times \Sigma^n\longrightarrow \mathbb R$
be its representing polymeasure. Suppose that $T_F(f_1, \ldots, f_m)=0$ whenever there exist $1\leq i, j\leq m$ such that $f_i$ and $f_j$ have disjoint support. Then, if we choose open sets $(G_1,\ldots, G_m)\in \Sigma^n\times \cdots \Sigma^n$, such that there exist $1\leq i, j\leq m$ with $G_i\cap G_j=\emptyset$ we have 
    $$\gamma(G_1,\ldots, G_m)
    =0.$$
\end{lema}

\begin{proof}
Given an open set $G_l\in \Sigma_n$, we can consider the directed set of the Borel
compact sets $C_l \subset G_l$ with the order given by the inclusion.
Applying Urysohn's lemma, for every such $C_l$ we can choose $f_{C_l}\in
C(S^{n-1})$, with $\chi_{C_l}\leq f_{C_l}\leq \chi_{G_l}$. It follows from the
regularity of the measures representing $C(S^{n-1})^*$ that the net
$f_{C_l}$ converges weak$^*$ to $\chi_{G_l}$. Hence, as explained in Theorem \ref{representacion}, 
    $$\gamma(G_1,\ldots, G_m)=\lim_{C_1}\cdots \lim_{C_m} T_F(f_{C_1}, \dots, f_{C_m})=0.$$
\end{proof}

Our purpose is to characterize the dual mixed volume by functional properties of $F$. The conditions above clearly do not suffice for this. In Theorem \ref{main} below we give conditions which do suffice. Condition (1) below appeared already in \cite{DuGaPe}. Condition (2) is new. This condition is based on the behaviour of the associated polynomial, which in this case turns out to be the $n$-dimensional volume, rather than the multilinear function. Let us note that this new condition is satisfied in case the polynomial is a valuation on the star bodies. This could our result more useful for certain applications.

We state and prove the result for star bodies. It remains true for star sets, replacing {\em regular measure} by {\em bounded additive measure} in condition (3) below. The proof for the case of star sets are slightly simple and we omit it.

Our main result is the following: 

\begin{teo}\label{main}
Let $F:\mathcal S_0^m \longrightarrow \mathbb R$ be separately additive and positively homogeneous. Let us also assume that $F$ verifies one of the following conditions: 

\begin{itemize}

\item[(A)] $F=F^+-F^-$, where $F^+$ and $F^-$ are also separately additive and positively homogeneous. 

\item[(B)] $F$ is separately bounded, in the sense of Proposition \ref{primera}. 
\end{itemize}

Then, the following are equivalent:  

\begin{enumerate}
\item $F(L_1,\ldots, L_m)=0$ whenever there exist $1\leq i_1, i_2\leq m$ such that $L_{i_1}\cap L_{i_2}=\{o\}$. 

\item $F$ is symmetric and the associated polynomial $P_F$ verifies $P_F(L\tilde{+} M)=P_F(L) + P_F(M)$ whenever $L\cap M=\{o\}$

\item There exists a regular  measure $\mu: \Sigma_n\longrightarrow \mathbb R$ such that $$F(L_1,\ldots, L_m)=\int \rho_{L_1}(t)\cdots \rho_{L_m} (t) d\mu(t). $$

\end{enumerate}

Moreover, if $F$ is rotationally invariant, then there exists  a constant $c\in \mathbb R$ such that $F(L_1, \ldots, L_n)=c\tilde{V}(L_1, \ldots, L_n)$
\end{teo}

\begin{proof}
Clearly (3) implies (1) and (2). Let us see that (1) implies (2): 

In the presence of (A) or (B), Propositions \ref{primera} and \ref{segunda} imply that $$T_F:C(S^{n-1})\times \cdots \times C(S^{n-1})\longrightarrow \mathbb R$$ is continuous and can be represented by a separately regular polymeasure 
$$\gamma_F:\Sigma_n\times \cdots \times \Sigma_n \longrightarrow \mathbb R$$

Let us see that $\gamma_F$ verifies that, for every $(A_1,\ldots, A_m)\in  \Sigma_n\times \cdots \times \Sigma_n$, if there exist $1\leq j, l\leq m$ such that $A_j\cap A_l=\emptyset$, then $\gamma_F(A_1,\ldots, A_m)=0$. 

Let us suppose without loss of generality that $A_1\cap A_2=\emptyset$, and let us choose  $\epsilon>0$. We apply  \cite[Lemma 2.5]{OA} and obtain compact sets  $C_i\subset A_i$ ($1\leq i \leq m$)  such that $$|\gamma_F(A_1,\ldots, A_m)-\gamma_F(C_1, \ldots, C_m)|<\epsilon.$$ We apply again \cite[Lemma 2.5]{OA} and the normality of $S^{n-1}$ to obtain open sets  $G_i\supset C_i$  ($1\leq i \leq m $), with $G_1\cap G_2=\emptyset$ such that

$$|\gamma_F(G_1,\ldots, G_m)-\gamma_F(C_1, \ldots, C_m)|<\epsilon.$$

Now, applying Lemma \ref{continuas} we get 

$$|\gamma_F(A_1,\ldots, A_m)|<|\gamma_F(C_1, \ldots, C_m)|+\epsilon <2\epsilon. $$ 

Since this happens for arbitrary $\epsilon>0$,  we get that $\gamma_F(A_1,\ldots, A_m)=0$. 

To see that $T_F$ is symmetric, it suffices to check that $\overline{T}_F:B(\Sigma_n)\times \cdots \times B(\Sigma_n)\longrightarrow 0$ is symmetric. By density, it is enough to check symmetry for simple functions. To see this, pick simple functions $g_1, \ldots, g_m$. There exist a finite collection of disjoint sets $(A_l)_{l=1}^s$ such that, for every $1\leq i \leq m$, $g_i=\sum_{l=1}^s a_l^i \chi_{A_l}$. Therefore $$\overline{T_F}(g_1, \ldots, g_m)=\overline{T_F}(\sum_{l_1=1}^s a_{l_1}^1 \chi_{A_{l_1}},\ldots, \sum_{l_m=1}^s a_{l_m}^m \chi_{A_{l_m}})=$$ $$=\sum_{l_1=1}^s \cdots \sum_{l_m=1}^s a_{l_1}^1\cdots a_{l_m}^m \gamma_F(\chi_{A_{l_1}}, \ldots, \chi_{A_{l_m}}).$$

We can use the previous reasonings to cancel all the terms where the $A_{l_i}$ do not all coincide, and we get 
$$\overline{T_F}(g_1, \ldots, g_m)=\sum_{l=1}^s a_{l}^1\cdots a_{l}^m \gamma_F(\chi_{A_{l}}, \ldots, \chi_{A_{l}})$$
and this expression is clearly symmetric. 

Finally, let $L, M$ be radial bodies verifying $L\cap M=\{o\}$. Then 
$$P_F(M\tilde{+}L)=T_F(L,L,\cdots, L)+ T_F(L,\ldots, L, M) + \cdots + T_F(M,M,\ldots, M)= $$ $$=T_F(L,L,\cdots, L)+T_F(M,M,\ldots, M)=P_F(L)+P_F(M),$$
and (2) follows. 

\smallskip

To see that  (2) implies (3), note first that, reasoning as before, we can assure the existence of   $T_F$ and $\gamma_F$. Condition (2) now implies that the polynomial associated to $T_F$ is orthogonally additive. Now, Theorem \ref{teoOA} suffices to finish.  

\smallskip

Finally, if $F$ is rotationally invariant we proceed as follows. Let $A\subset S^{n-1}$ be a Borel set, let $\phi$ be a rotation on $S^{n-1}$ and let  $S=\{f\in C(S^{n-1})\ : \ f\leq \chi_A \}$. The regularity of  $\mu $ implies that  $\mu(A)=\sup_{f\in S}\int fd\mu$. Let us note that $\phi S=\{\phi f\ :\ f\in S\}=\{g\in C(S^{n-1})\ : \ g\leq \chi_{\phi A})\}$. Thus, using the rotational invariance of $F$, we get

$$\mu(A)=\sup_{f\in S}\int f d\mu=\sup_{f\in S}\int \phi f d\mu=\sup_{\phi f\in \phi S}\int \phi f d\mu=\sup_{g\in \phi S}\int g d\mu\leq \mu(\phi(A)),$$
where $g=\phi(f)$. 
%
%
%
%
%
We can reason similarly to obtain $\mu(\phi A)\leq \mu(A)$. This, together with the uniqueness (up to constant multiplication) of the Lebesgue measure among rotation invariant measures on $S^{n-1}$, concludes the proof.
\end{proof}

Let us note that we need some condition guaranteeing the continuity of $T_F$. Consider a sequence $(x_i)_{i\in \mathbb N} \subset S^{n-1}$, with $x_i\not = x_j$ for every $i\not = j$. Consider now a sequence of disjoint open sets $G_i\supset x_i$ and positive functions $f_i\in C(S^{n-1})^+$ verifying $\|f_i\|=1$, $f_i(x_i)=1$ and $supp f_i\subset G_i$.  Clearly the set $\{f_i; i\in \mathbb N\}\subset C(S^{n-1})$ is linearly independent. Therefore, it can be completed to a Hamel basis $B$ of $C(S^{n-1})$. Consider now the bilinear form $$T:C(S^{n-1})\times C(S^{n-1})\longrightarrow \mathbb R$$ defined on the basis $B$ by $$T(f_i, f_j)=\delta_{ij}, \, 
T(b,b')=0 \mbox{ for any other choice of }b, b'\in B,$$ and $T$ defined by linearity on the rest of $C(S^{n-1})\times C(S^{n-1})$. 

Consider now the function $F:K_0^2\longrightarrow \mathbb R$ defined by $F(L,M)=T(\rho_L, \rho_M)$. Then $F$ is linear, but it can not be represented by a polymeasure $\gamma$.

\smallskip

The result above includes and extends the main results in \cite{DuGaPe}. To see this, note that, as the authors mention in that paper, additive functions taking values in $[0,\infty)$ are always positively homogenenous.

\smallskip

In that same paper, the authors study the following question 

\begin{question}
Let $F:\mathcal S_0^n \longrightarrow \mathbb R$ be separately additive, rotation invariant, and vanishes when the intersection of two of the arguments is $\{o\}$. Does there exists a constant $c\in \mathbb R$ such that $F(L_1, \ldots, L_n)=c\tilde{V}(L_1, \ldots, L_n)$ ?

\end{question}

They show that, thus formulated, in the presence of the Axiom of Choice the question is false. The counterexample follows from the existence of additive non linear functions. Next, they ask if a positive answer to the question is compatible with ZF. A path to answer this could be to study if the reasonings in \cite{OA} apply in the Solovay model, since in that case additive functions are automatically linear and continuous.

\end{document}